\documentclass[12pt, reqno]{amsart}

\usepackage{amsmath,amsthm}
\usepackage{hyperref}

\def\gcd{\operatorname{gcd}}
\def\ZZ{\mathbb{Z}}
\def\aa{\underline{a}}
\def\jj{\underline{j}}

\newtheorem{lemma}{Lemma}[section]
\newtheorem{corollary}[lemma]{Corollary}
\newtheorem{theorem}[lemma]{Theorem}
\newtheorem{proposition}[lemma]{Proposition}

\newtheorem{example}[lemma]{Example}

\def\phi{\varphi}

\def\A{\mathbb A}

\def\mm{{\bf m}}

\advance\headheight1.15pt
\begin{document}
\title{Periodic Occurence of Complete Intersection Monomial Curves}
\author{A. V. Jayanthan}
\address{Department of Mathematics, Indian Institute of Technology Madras, Chennai, India -- 600036.}
\email{jayanav@iitm.ac.in}
\author{Hema Srinivasan}
\address{Department of Mathematics, University of Missouri-Columbia, Columbia, MO, USA -- 65211.}
\email{srinivasanh@missouri.edu}
\begin{abstract}
  We study the complete intersection property of monomial curves in
  the family $\Gamma_{\aa + \jj} = \{(t^{a_0 + j}, t^{a_1+j}, \ldots,
  t^{a_n + j}) ~ | ~ j \geq 0, ~ a_0 < a_1 < \cdots < a_n \}$. We prove
  that if $\Gamma_{\aa+\jj}$ is a complete intersection for $j \gg0$,
  then $\Gamma_{\aa+\jj+\underline{a_n}}$ is a complete intersection
  for $j \gg 0$.  This proves a conjecture of Herzog and Srinivasan on eventual periodicity of Betti numbers of semigroup rings under translations for complete intersections. 
   We also show that if $\Gamma_{\aa+\jj}$ is a complete
  intersection for $j \gg 0$, then $\Gamma_{\aa}$ is a complete
  intersection. We also characterize the complete intersection
  property of this family when $n = 3$.        \end{abstract}
\maketitle
\section*{Introduction}
Given an ascending  sequence of positive integers $e_0,\ldots, e_n$
the curve in $\A^{n+1}$ parameterized by $t\to (t^{e_0}, t^{e_1}\ldots
,t^{e_n})$  is called an affine monomial curve since the
parametrization is by monomials.    The minimal number of equations
defining   monomial curves and the various structures of monomial
curves have been fascinating algebraists and  geometors for a long
time.  It is well known that these equations are binomial equations.
In fact,  the ideal of a monomial curve in $\A^{n+1}$ is a weighted
homogeneous binomial prime ideal of height $n$ in the polynomial ring
$R= k[x_0, \ldots x_n]$ . In the plane, they are principal ideals and the
space monomial curves are either complete intersections, generated by
two binomials or determinantal ideals generated by the $2\times 2$
minors of a $2\times 3$ matrix \cite{herzog}.   This breaks down even
in dimension 4 because there is no upper bound for the number of
generators for monomial cuves in $\A^4$ \cite{BH}.  However,  because
of the structure theorem of Gorenstein ideals in Codimension three as
ideals generated by Pfaffians \cite{BH}, the Gorenstein monomial
curves in dimension three are either complete intersections, generated
by 3  elements or the ideal of $4\times 4$ pfaffians of a $5\times 5$
skew symmetric matrix.   Thus, for special classes of monomial curves,
the number of generators are bounded.    We partition the monomial curves
in to classes so that two monomial curves are in the same class if their
consecutive parameters have the same differences.   That is, if 
$\mm = \{m_1,\ldots, m_n\}$ is a  sequence of positive integers,
$C(\mm)$ is a class of monomial curves  defined by ${\aa} = a_0,
\ldots, a_n$ with $\Delta ({\aa}) = \mm$.    Herzog and Srinivasan
conjecture that the minimal number of generators for the ideal
defining the monomial curves in a given class $C(\mm)$ is bounded.  In
fact, they conjecture that this is eventually periodic with period
$\sum_i  m_i$. This conjecture is true for monomial curves defined
by arithmetic sequences \cite{GSS2}.  In this paper we prove the
conjecture in dimension 3 completely and prove it for complete
intersections in any dimension.  We prove that for $a_0\gg 0$, the
complete intersections in the class $C(\mm)$ occur periodically with
period $\sum_i m_i$.  Our proof of the conjecture follows from a
criterion for complete intersection extending the one in \cite{D} for monomial curves with high $a_0$. This generalizes and recovers some results of Adriano
Marzullo, \cite{Ad}.
 
Now we state the conjecture precisely:  Let $\aa= (a_0, \ldots
a_n)$ be a  sequence of positive integers and let $j$ be any positive
integer. Let $\aa + (\jj)$ denotes sequence
$(j+a_0, j+a_1, j+a_1, \ldots, j+a_n)$. Let $\Gamma_{\aa + (\jj)}$
denote the monomial curve corresponding to the sequence $\aa + (\jj)$
and $I_{\aa+(\jj)}$ denote the defining ideal of $\Gamma_{\aa+(\jj)}$.
Then the strong form of
Herzog-Srinivasan  conjecture states that the Betti numbers of $I_{\aa +
(\jj)}$ are eventually periodic in $j$.    Thus, the conjecture says
that within a class of monomial curves associated to  increasing
sequences $\aa $ with the same $\Delta {\aa}$, the Betti numbers of
the defining ideals are eventually periodic in $a_0$.  

In this paper we prove that for any sequence $\aa$, for large $j$,  if
$\aa + (\jj)$ is a complete intersection, then, $\aa +
(\underline{j+a_n})$ is a complete intersection. Since we are proving
results for $j \gg 0$, we may as well assume that $a_0 = 0$ and the
sequence $\aa + (\jj) = (j, j+a_1, \ldots, j+a_n)$.  To be precise,
let $CI(\aa)=\{j~ |~ \Gamma_{\aa+(\jj)} \mbox{ is a complete 
intersection curve}\}$.  We prove:

\vskip 2mm
\noindent
{\bf Theorem \ref{main}} {\em If $\aa = (a_1, \ldots, a_n)$,
then $CI(\aa)$ is either finite or eventually periodic with period
$a_n$. If for $j \gg 0$, $\aa+(\jj)$ is a complete intersection, then
there exists $1 \leq t \leq n-1$ and $k \in \ZZ_+$ such that 
\begin{enumerate}
  \item[(a)] $j = a_nm$ for some $m \in \ZZ_+$, 
  \item[(b)] $\gcd(a_1, \ldots, a_{t-1}, a_t+a_{t+1}, a_{t+2}, \ldots,
	a_n) = k \neq 1$ and
  \item[(c)] $\displaystyle{ a_t \in \left \langle \frac{a_1}{k},
	\ldots, \frac{a_{t-1}}{k}, \frac{a_{t+1}}{k}, \ldots,
	\frac{a_n}{k} \right\rangle}$.
\end{enumerate}}
\vskip 2mm
\noindent
and 
\vskip 2mm
\noindent
{\bf Corollary \ref{cor1}}
{\em For $j \gg 0$, if $\aa + (\jj)$ is a complete intersection, then $\aa$
is a complete intersection. }

\vskip 2mm
\noindent
We also give a criterion for these curves to be a complete 
intersections in Theorem \ref{crit}. 

\section{Monomial curves in $\mathbb{A}^3$}
Let $\aa = (a_0,a_1, \ldots, a_n) \in \ZZ_+^{n+1}$ with $a_0< a_1 < \cdots <
a_n$ and $R = k[t^{a_0}, \ldots, t^{a_n}]$, where $k$ is a field of
characteristic zero. We say that $\aa$ is a complete intersection
sequence if $R$ is a complete intersection. 
For the reason explained in the introduction,
we will assume here that $a_0 = 0$. 
We begin by recalling a result characterizing the
complete intersection property of the sequence $\aa$. For any sequence
$\aa$, let $\langle a_1, \ldots, a_n \rangle := \{\sum_{i=1}^n r_ia_i
~ | ~ r_i \in \ZZ_{\geq 0} \}$ be the semigroup generated by $a_1, \ldots,
a_n$.

\begin{theorem}\label{complete-int}{\em [Proposition 9, \cite{D}]}
The sequence $\aa$ is a complete intersection if and only if $\aa$ can
be written as a disjoint union:
$$
\aa = k_1(b_{i_1}, \ldots, b_{i_r}) \sqcup k_2(b_{i_{r+1}}, \ldots,
b_{i_n}),
$$
where $a_{i_m} = k_1b_{i_m}$ for $m = 1, \ldots, r$, $a_{i_m} = k_2
b_{i_m}$ for $m = r+1, \ldots, n$,
$\gcd(k_1, k_2) = 1$, $k_1 \notin \{b_{i_{r+1}}, \ldots, b_{i_n}\}$,
$k_1 \in \langle b_{i_{r+1}}, \ldots, b_{i_n} \rangle$, $k_2 \notin
\{b_{i_1}, \ldots, b_{i_r}\}$, $k_2 \in \langle \{b_{i_1}, \ldots,
b_{i_r} \rangle$ and both $(b_{i_1}, \ldots, b_{i_r})$ and
$(b_{i_{r+1}}, \ldots, b_{i_n})$ are complete intersection sequences.
\end{theorem}

We say that the sequence $\aa$ is a complete intersection of type $(r,
n-r)$ if it splits as in the above theorem. 

\begin{lemma}\label{split-type1}
Suppose $j > a_n^2$. If $(j, j+a_1, j+a_2, \ldots, j+a_n)$ is a
complete intersection of the type $(m, n+1-m)$, then either $m = 1$ or
$m = n$.
\end{lemma}
\begin{proof}
Suppose $1< m < n$, then we have a split of the form
$$
(j, j+a_1, \ldots, j+a_n) = k_1(\alpha_1, \ldots, \alpha_m)
\sqcup k_2(\alpha_{m+1}, \ldots, \alpha_{n+1}),
$$
where $k_1 \in \langle \alpha_{m+1}, \ldots, \alpha_{n+1} \rangle$ and
$k_2 \in \langle \alpha_1, \ldots, \alpha_m \rangle$.
Since $k_i$ divides $j+a_{l}$ for $1 < l < n$, $k_i \leq
a_n$. Since $j > a_n^2$ and $k_2 \leq
a_n$, $\alpha_j > a_n$. This contradicts the
fact that $k_1 \in \langle \alpha_{m+1}, \ldots, \alpha_{n+1}
\rangle$. Therefore $m = 1$ or $m = n$.
\end{proof}

\begin{lemma}\label{split-type2}
Suppose $\aa + (\jj)$ is a complete intersection sequence for $j \gg 0$.
Then complete intersection splits of the type
\begin{enumerate}
  \item $\aa+(\jj) = k_1\left(\frac{j}{k_1}\right) \sqcup
	k_2\left(\frac{j+a_1}{k_2}, \ldots,
	\frac{j+a_n}{k_2}\right)$;
  \item $\aa+(\jj) = k_1\left(\frac{j}{k_1}, \frac{j+a_1}{k_1}, \ldots,
	\frac{j+a_{n-1}}{k_1}\right) \sqcup
	k_2\left(\frac{j+a_n}{k_2}\right)$
\end{enumerate}
are not possible.
\end{lemma}
\begin{proof}
First we prove that a split as in $(1)$ is not possible. Suppose $(1)$
is a complete intersection split of $\aa+(\jj)$. First note that $k_2$
divides $a_i$ for $i \geq 2$. If $k_1 \neq j$, then by multiplying by
an appropriate factor, we obtain
\begin{eqnarray*}
j & = & \alpha_1 \frac{j+a_1}{k_2} + \cdots + \alpha_n
\frac{j+a_n}{k_2} \\
& = & (\alpha_1 + \cdots + \alpha_n)\frac{j}{k_2} + \alpha_1
\frac{a_1}{k_2} + \alpha_2 \frac{a_2}{k_2} + \cdots + \alpha_n
\frac{a_n}{k_2},
\end{eqnarray*}
where $\alpha_i$'s are non-negative integers.
Therefore, 
$$k_2 j = (\sum_{i=1}^n\alpha_i)(j) + \alpha_1a_1 +
\cdots + \alpha_na_n$$ 
so that 
$$
[k_2 - (\alpha_1+\cdots+\alpha_n)]j  = \alpha_1a_1 +
\cdots + \alpha_na_n
$$
Since the right hand side consists of linear combination of
non-negative integers, not all of them zero, $k_2 > \sum_{i=1}^n
\alpha_i$. Therefore, we have 
\begin{eqnarray*}
 0 < [k_2 - (\alpha_1+\cdots+\alpha_n)]j & < &
  (\sum_{i=1}^n\alpha_i)(a_n) \leq  
  k_2 (a_n)  \leq a_n^2.
\end{eqnarray*}
This contradicts the fact that $j > a_n^2$. Therefore a
split of the first kind is not possible.

\vskip 2mm \noindent
Now assume that $(2)$ is a complete intersection split for $\aa+ (j)$,
for $j > a_n^2$. If
$k_2 \neq j+\sum_{i=1}^na_i$, then after multiplying with an
appropriate factor we get
$$
j+a_n = \alpha_1 \frac{j}{k_1} + \cdots + \alpha_n
\frac{j+a_{n-1}}{k_1},
$$
where $\alpha_i$'s are non-negative integers. Therefore

\begin{eqnarray}\label{eqn1}
a_n & = &
\left(\sum_{i=1}^n\alpha_i-k_2\right)\frac{j}{k_2} +
\alpha_2\frac{a_1}{k_2}+ \alpha_3 \frac{a_2}{k_2} + \cdots
+\frac{\alpha_n}{k_2} a_{n-1}.
\end{eqnarray}

If $\sum_{i=1}^n \alpha_i < k_2$, then it follows from the above
equality that
\begin{eqnarray*}
  a_n & \leq &
  \left(\sum_{i=1}^n\alpha_i-k_2\right)\frac{j}{k_2} + \frac{1}{k_2}
  a_n\left(\sum_{i=1}^n \alpha_i\right) \\
  & \leq & \left(\sum_{i=1}^n\alpha_i-k_2\right)\frac{j}{k_2} +
  a_n \leq 0.
\end{eqnarray*}
The last inequality holds since $k_2 \leq \sum_{i=1}^n a_i$ and $j >
a_n^2$. This is a contradiction, since the left hand side of the
inequality is a positive integer.

Now suppose $\sum_{i=1}^n \alpha_i > k_2$. It follows from equation
(\ref{eqn1}) that
\begin{eqnarray*}
  k_2a_n  & = & 
  \left(\sum_{i=1}^n\alpha_i-k_2\right)j + \alpha_2 a_1
  + \cdots + \alpha_na_{n-1}.
\end{eqnarray*}
Therefore $j \leq k_2 a_n < a_n^2$, which is a contradiction to the
hypothesis that $j > a_n^2$.

If $\sum_{i=1}^n \alpha_i = k_2$, then we have
\begin{eqnarray*}
  a_n & = & \alpha_2 a_1
  + \cdots + \alpha_na_{n-1}\\
  & < & \frac{1}{k_2} \left(\sum_{i=1}^n \alpha_i\right)
  a_{n-1} = a_{n-1}.
\end{eqnarray*}
This is again a contradiction. Therefore, 
all three possibilities lead to contradiction. Hence
a complete intersection split of type (2) is not possible.
\end{proof}

We now prove the periodicity conjecture for monomial curves in 
$\mathbb{A}^3$.
Let $a_1 = a$ and $a_2 = b$. We first prove a characterization for
$(j, j+a, j+b)$ to be a complete intersection sequence for $j \gg
0$.
\begin{theorem}\label{complete-int1}
If $j \geq \max\{ab, b(b-a)\}$, then $(j,j+a, j+b)$ defines a
complete intersection ideal if and only if there exist $(j, b) = k
\neq 1$ and non-negative integers $\alpha, \beta$ such that $k(j+a) =
\alpha (j)+\beta (j+b)$. Moreover, in this case, $\alpha +\beta =
k$ and $(a,b-a) = s$ with $b = sk$. 
In partiular, if $a$ and $b-a$ are relatively prime,  $(j,j+a,j+b)$
is a complete intersection if and only if $b$ divides $j$.
\end{theorem}
\begin{proof}
Let $(j, j+a, j+b)$ be a complete intersection sequence. By Theorem
\ref{complete-int} and Lemma \ref{split-type2}, we can have only one
split possible, namely: 
\vskip 2mm \noindent
$$
(j, j+a, j+b) = \frac{j+a}{k'}(k') \sqcup k\left(\frac{j}{k},
\frac{j+b}{k}\right),
$$
where $k' \mid k,$ $\gcd\left(\frac{j+a}{k'}, k\right) = 1$ and
$\frac{j+a}{k'} \in \left\langle \frac{j}{k}, \frac{j+b}{k}
\right\rangle$. Let $\alpha, \beta$ be non-negative integers such that
$k(j+a) = \alpha j+\beta (j+b)$.  Since $k\leq b$, we see that
$k(j+a) \le kj+j = j(k+1)$.  Therefore $kj+ka = j(\alpha
+\beta)+\beta b$ so that $\alpha +\beta \le k$. If $\alpha+\beta <
k$, then the equation $ka = (\alpha+\beta - k) j + \beta b$ would
imply that $ka < 0$ if $j \gg 0$. Therefore, $\alpha+\beta = k$.

Further,  in this even, $\alpha a  = \beta (b-a)$.  Therefore $b = (\alpha
+\beta )s$ so that $\alpha(b) = (b-a)(\alpha+\beta) = \alpha s (\alpha
+\beta)$.  Hence $b-a =\alpha s$ and $a = \beta s$.  
 
If $\gcd(a,b-a)=1$, then $s =1$ and hence $\alpha +\beta = k = b$,
there by establishing that $b$ divides $j$.
 
 The converse is clear.  
\end{proof}
We new prove the periodicity conjecture for $n = 2$.
\begin{theorem}
Let $\aa + (\jj) = (j, j+a, j+b)$ and let $I_{\aa + (\jj)}$
denote the defining ideal of the monomial curve $(t^j, t^{j+a},
t^{j+b})$. If $j \gg 0$, then the betti numbers of $I_{\aa+(\jj)}$
are periodic for with period $b$.
\end{theorem}
\begin{proof}
Since the ideals $I$ in this case are either complete intersections or
height $2$ Cohen-Macaulay ideals generated by $3$ elements, we simply
need to show the periodicity of the number of generators.  
   
By Theorem \ref{complete-int1}, if this is a complete intersection,
then $(j,b) = k,~ k(j+a) = \alpha j+\beta (j+b)$, with $\alpha
+\beta = k$ and $\gcd(j, b) = k$.  Thus, $ \alpha (j+b)+\beta
(j+2b) = k(j+a) +(\alpha +\beta)b= k(j+a+b)$.  Therefore,
$(j+b, j+a+b, j+2b)$ also defines a complete intersection.  
   
Conversely, Suppose  $(j+a+b, j+a+b, j+2b)$ defines a complete
intersection. Since $j\ge \max\{ab,b(b-a)\}$, we have the same 
$\alpha, \beta$ giving the equations as before. Therefore, for $j \ge
\max\{ab,b(b-a)\},~  (j+rb,j+a+rb,j+b+rb)$ is a
complete intersection for all $r$ if and only if it is a complete 
intersection for $(j, j+a, j+b)$.  
Thus the eventual periodicity is true for $d=2$. 
\end{proof}

\section{Monomial curves in $\mathbb{A}^n$ for $n \geq 4$}
In this section we prove the periodicity of occurence of complete
intersections in the class $\Gamma_{\aa+(\jj)} \subset \mathbb{A}^n$
and characterize complete intersection monomial curves
in $\mathbb{A}^4$.
\begin{theorem}\label{main}
If $\aa = (a_1, \ldots, a_n)$, then $CI(\aa)$ is either finite or
eventually periodic with period $a_n$. If for $j \gg 0$, $\aa+(j)$ is
a complete intersection then there exists $1 \leq s \leq n-1$ and $k
\in \ZZ_+$ such that 
\begin{enumerate}
  \item[(a)] $j = a_nm$ for some $m \in \ZZ_+$, 
  \item[(b)] $\gcd(a_1, \ldots, a_{s-1}, a_s+a_{s+1}, a_{s+2}, \ldots,
	a_n) = k \neq 1$ and
  \item[(c)] $\displaystyle{ a_s \in \left \langle \frac{a_1}{k},
	\ldots, \frac{a_{s-1}}{k}, \frac{a_{s+1}}{k}, \ldots,
	\frac{a_n}{k} \right\rangle}$.
\end{enumerate}
\end{theorem}
\begin{proof}
We first assume that $\gcd(a_1, \ldots, a_n) = 1$.
Assume that $CI(\aa)$ is not finite. 
Assume
that $\aa+(\jj)$ is a complete intersection. Therefore it follows from
Lemma \ref{split-type1} and Lemma \ref{split-type2}, that  we have the
complete intersection split of the form
\begin{eqnarray*}\label{eqn2}
\aa+(\jj) & = &  \frac{j+a_s}{k'} (k') \sqcup k
\left(\frac{j}{k}, \frac{j+a_1}{k}, \ldots,
\frac{j+a_{s-1}}{k}, \frac{j+a_{s+1}}{k},
\ldots, \frac{j+a_n}{k}\right)
\end{eqnarray*}
which satisfies the conditions in Theorem \ref{complete-int}. 
Furthermore, we have the following:
\begin{enumerate}
  \item $k' \mid k$ and $k' \neq k$. 
  \item $k \mid \gcd(a_1, \ldots, a_{s-1}, a_{s+1}, \ldots, a_n)$.
  \item Since $k' \mid k$, it divides $a_i$ for $i = 1, \ldots,
	s-1$ and it divides $j$ as well. Therefore, $k' \mid a_s$ and
	hence $k' \mid a_i$ for all $i = 1, \ldots, n$. This implies
	that $k' \mid \gcd(a_1, \ldots, a_n) = 1$. Therefore $k' =
	1$.
\end{enumerate}

\noindent
Since $\left(\frac{j}{k}, \frac{j+a_1}{k}, \ldots,
\frac{j+a_{s-1}}{k}, \frac{j+a_{s+1}}{k},
\ldots, \frac{j+ a_n}{k}\right)$ is a complete
intersection sequence (associated to a sequence of length $n-1$), it
follows by induction on $n$ that for $j \gg 0$, 
$$
\frac{j}{k} = \frac{a_n}{k}m
$$
and hence $j = a_nm$, where $m$ is a positive integer. Since $k' = 1$,
we have
$$j+a_s \in \left\langle
\frac{j}{k}, \frac{j+a_1}{k}, \ldots,
\frac{j+a_{s-1}}{k}, \frac{j+a_{s+1}}{k},
\ldots, \frac{j+ a_n}{k} \right\rangle,$$
and therefore
there exist some non-negative integers $\alpha_1, \ldots, \alpha_n$,
not all zero such that
\begin{eqnarray}\label{eqn3}
j+\sum_{i=1}^s a_i = \alpha_1 \frac{j}{k} + \cdots + \alpha_s
\frac{j+a_{s-1}}{k} + \alpha_{s+1} \frac{j+a_{s+1}}{k} + \cdots + \alpha_n
\frac{j+a_n}{k}.
\end{eqnarray}
\textsc{Claim 1:} $\sum_{i=1}^n \alpha_i = k$.
\vskip 2mm \noindent
\textit{Proof of Claim 1:} From the above equation, we can write
\begin{eqnarray*}
  a_s & = & \left(\sum_{i=1}^n \alpha_i - k\right)
  \frac{j}{k} + \sum_{l=1}^s \alpha_l\frac{a_{l-1}}{k} + \sum_{l=s+1}^n \alpha_l
  \frac{a_l}{k}.
\end{eqnarray*}
Suppose $\sum_{i=1}^n \alpha_i > k$. If $j > a_n^2$,
 then $\frac{j}{k} > a_n$ and hence we get
that $a_s > a_n$, a contradiction.

\vskip 2mm \noindent
Suppose $\sum_{i=1}^n \alpha_i < k$. Then we get
\begin{eqnarray*}
  a_t & \leq & \left(\sum_{i=1}^n \alpha_i - k\right)
  \frac{j}{k} + \frac{1}{k}\left(\sum_{i=1}^n \alpha_i\right)
  a_n \\
  & < & \left(\sum_{i=1}^n \alpha_i - k\right)
  \frac{j}{k} + a_n \leq 0,
\end{eqnarray*}
where the last inequality holds since $\frac{j}{k} \geq \sum_{i=1}^n
a_i$. This is contradiction since $a_s > 0$. Therefore,
we have shown that neither of the cases
\begin{enumerate}
  \item[(a)] $\sum_{i=1}^n \alpha_i > k$
  \item[(b)] $\sum_{i=1}^n \alpha_i < k$
\end{enumerate}
are not possible. Therefore, $\sum_{i=1}^n \alpha_i = k$. This
completes the proof of the claim.

\vskip 2mm \noindent
\textsc{Claim 2:} $\aa + (\underline{j+a_n}) = \left(j+a_n, j+a_n + a_1, \ldots,
j+ 2a_n \right)$ is a complete intersection sequence.

\vskip 2mm \noindent
\textit{Proof of Claim 2:} We show that this sequence has a complete
intersection split similar to that of $\aa_j$. Choose $\alpha_1,
\ldots, \alpha_n$ as in (\ref{eqn3}). Therefore we have $\sum_{i=1}^n
\alpha_i = k$ so that 
\begin{eqnarray*}
  \sum_{l=1}^s\alpha_l\left(\frac{j}{k} + \frac{a_n}{k} + 
  \frac{a_{l-1}}{k}\right) & + &  
  \sum_{l=s+1}^n \alpha_l \left(\frac{j}{k} + \frac{a_n}{k} + 
  \frac{a_l}{k}\right) \\
  & = & j+a_s + \left(\sum_{i=1}^n \alpha_i \right)
  \frac{a_n}{k}\\
  & = & j+a_s+ a_n
\end{eqnarray*}
Therefore, $$
j + a_n + a_s \in
\left\langle \frac{j}{k}+\frac{a_n}{k}, \ldots, 
\frac{j}{k} +
\frac{a_{s-1}}{k}+ \frac{a_n}{k},
\frac{j}{k} +
\frac{a_{s+1}}{k}+ \frac{a_n}{k},\ldots
\frac{j}{k} +
\frac{a_n}{k}+ \frac{a_n}{k} \right\rangle,
$$
We need to show that this split satisfies
all the properties of Theorem \ref{complete-int}. Since
$\gcd(j+a_s, k) = 1$ and $k \mid a_n$,
$\gcd(j+a_n + a_s, k) = 1$. Note that
$\gcd\left(\frac{a_1}{k}, \ldots, \frac{a_{s-1}}{k},
\frac{a_{s+1}}{k}, \ldots, \frac{a_n}{k}
\right) = 1$. Since 
$ 
\left(\frac{j}{k}, \frac{j+a_1}{k}, \ldots,
\frac{j+a_{s-1}}{k}, \frac{j+a_{s+1}}{k},
\ldots, \frac{j+ a_n}{k}\right)
$ 
is a complete intersection, by induction on $n$, we get that
$$
\left( \frac{j}{k}+\frac{a_n}{k}, \ldots, 
\frac{j}{k} +
\frac{a_{s-1}}{k}+ \frac{a_n}{k},
\frac{j}{k} +
\frac{a_{s+1}}{k}+ \frac{a_n}{k},\ldots,
\frac{j}{k} +
\frac{a_n}{k}+ \frac{a_n}{k} \right)
$$
is a complete intersection sequence. Therefore, if $CI(\aa)$ is
infinite, then it is eventually periodic with period $a_n$.

Now let $k' = \gcd(a_1, \ldots, a_n)$.
Assume that $\aa + (\jj)$ is a complete intersection.
Then we have a split of the form:
$$
\aa + (\jj) = \frac{j+a_s}{k'} ~ (k') \sqcup k
\left(\frac{j}{k}, \ldots, \frac{j+a_{s-1}}{k},
\frac{j+a_{s+1}}{k}, \ldots, \frac{j+a_n}{k}
\right),
$$
where $k = \gcd(a_1, \ldots, a_{s-1}, a_{s+1}, \ldots,
a_n)$, $k' \mid k$, $k \neq k'$,
$\gcd\left(\frac{j+a_s}{k'}, k\right) = 1$ and \\
$\left(\frac{j}{k}, \ldots, \frac{j+a_{s-1}}{k},
\frac{j+a_{s+1}}{k}, \ldots, \frac{j+a_n}{k}
\right)$ is a complete intersection. Since $\gcd\left(\frac{a_1}{k},
\ldots, \frac{a_{s-1}}{k}, \frac{a_{s+1}}{k}, 
\ldots, \frac{a_n}{k}\right) = 1$, we can use the Theorem \ref{main}
to conclude that $\frac{j}{k} = \frac{a_n}{k} m$ for some
$m \in \ZZ_+$. Hence $j = (\sum_{i=1}^n a_i)m$. We also have that\\
$\displaystyle{\left(\frac{j+a_n}{k}, \ldots,
\frac{j+a_n+a_{s-1}}{k},
\frac{j+a_n+a_{s+1}}{k}, \ldots,
\frac{j+a_n+a_n}{k}
\right)}$  is a complete intersection, since the period being
$\frac{a_n}{k}$. This shows that we have a complete
intesection split 
$$\aa+(\underline{j+a_n}) = 
\frac{j+a_n+a_s}{k'} ~(k') \sqcup
k \left(\frac{j+a_n}{k}, \ldots, \frac{j+a_n+a_{s-1}}{k},
\frac{j+a_n+a_{s+1}}{k}, \ldots,
\frac{j+a_n+a_n}{k}\right).$$
\vskip 2mm \noindent
This implies that $\aa + (\underline{j+a_n})$ is a complete 
intersection, proving the periodicity as well.
\end{proof}
As a consequence of the above result, we relate the complete
intersection property of $\aa$ and $\aa + (\jj)$ for $j \gg 0$.
\begin{corollary}\label{cor1}
For $j \gg 0$, if $\aa + (\jj)$ is a complete intersection, then $\aa$
is a complete intersection. 
\end{corollary}
\begin{proof}
First assume that $\gcd(a_1, \ldots, a_n) = 1$.
We prove the first statement by induction on $n$. If $n = 1$, there is
nothing to prove and for $n = 2,$ $(a_1, a_2)$ is always a complete
intersection.  Assume now that $\aa = (a_1, \ldots, a_n)$ with $n \geq
3$ and $\aa + (\jj)$ is a complete intersection for $j \gg 0$.  By
Theorem \ref{main}, there exists a $s$ and $k$ such that 
$$
\aa + (\jj) = j+a_s ~ (1) \sqcup k \left(\frac{j}{k},
\frac{j+a_1}{k}, \ldots, \frac{j+a_{s-1}}{k}, \frac{j+a_{s+1}}{k},
\ldots, \frac{j+a_n}{k}\right)
$$
with $\left(\frac{j}{k}, \frac{j+a_1}{k}, \ldots, \frac{j+a_{s-1}}{k},
\frac{j+a_{s+1}}{k}, \ldots, \frac{j+a_n}{k}\right)$ a complete
intersection, $\gcd\left(j+a_s, k\right) = 1$
and $a_s \in \left\langle
\frac{a_1}{k}, \ldots, \frac{a_{s-1}}{k}, \frac{a_{s+1}}{k}, \ldots,
\frac{a_n}{k}\right\rangle$. By induction on $n$, we get that
$\left(\frac{a_1}{k}, \ldots, \frac{a_{s-1}}{k}, \frac{a_{s+1}}{k},
\ldots, \frac{a_n}{k}\right)$ is a complete intersection. Hence we
have a complete intersection split:
$$
\aa = a_s~(1) \sqcup k\left(\frac{a_1}{k}, \ldots,
\frac{a_{s-1}}{k}, \frac{a_{s+1}}{k}, \ldots, \frac{a_n}{k}\right).
$$
Therefore $\aa$ is a complete intersection. Now assume that $\gcd(a_1,
\ldots, a_n) = k'$. Since $\aa + (\jj)$ is a complete intersection for
$j \gg 0$, it follows from Theorem \ref{main} that $j = a_nm$ for some
$m \in \ZZ_+$. Therefore $k' \mid j$. Let $j' = \frac{j}{k'}$ and $a_i
= \frac{a_i}{k'}$. Since $\aa + (j)$ is a complete intersection so is
$(j', j'+a_1', \ldots, j'+a_n')$. By the first part, this implies that
$(a_1', \ldots, a_n')$ is a complete intersection and hence $(a_1,
\ldots, a_n)$ too is a complete intersection.
\end{proof}

We now prove a partial converse of the above corollary. It can be seen
that a converse statement of Corollary \ref{cor1} is not true, cf.
Example \ref{ex2}
\begin{proposition}
If $n \geq 3$ and $\aa$ is a complete
intersection and $k_{i+1}a_i \in \langle a_{i+1}, \ldots, a_n
\rangle$, where $k_i = \gcd(a_i, \ldots, a_n)$, then there exists $j
\gg 0$ such that $\aa + (\jj)$ is a complete intersection. 
\end{proposition}
\begin{proof}
First assume that $\gcd(a_1, \ldots, a_n) =
1$. We prove the assertion by induction on $n$. Let $n = 3$.
Let
$$
\aa = a_1~(1) \sqcup k\left(\frac{a_2}{k}, \frac{a_3}{k}\right),
$$
where $k = \gcd(a_2, a_3)$. Since $a_1 \in \left\langle \frac{a_2}{k},
\frac{a_3}{k}\right\rangle$, we can write $k a_1 = \beta a_2 + \gamma
a_3$. Since $a_1 < a_i$ for $i = 2, 3$, $k \geq \beta + \gamma$. Let
$\alpha = k - \beta - \gamma$. Then for $j \geq 0$, $(j + a_1) =
\alpha \frac{j}{k} + \beta \frac{j+a_2}{k} + \gamma \frac{j+a_3}{k}$.
By Theorem \ref{complete-int1}, there exists $j \gg 0$, $j = a_3m$
such that $\left(\frac{j}{k}, \frac{j+a_2}{k}, \frac{j+a_3}{k}\right)$
is a complete intersection.

Let $a_1 = \alpha_2 \frac{a_2}{k} + \cdots +
\alpha_n\frac{a_n}{k}$. Since $a_1 < a_i$ for all $i =
2, \ldots, n$, $\sum_{i=2}^n \alpha_i\leq k$. Let $\alpha_1 = k -
\sum_{i=2}^n \alpha_i$. Then for any $j > 0$, we can write
$$
j+a_1 = \alpha_1 \frac{j}{k} + \alpha_2 \frac{j+a_2}{k} +
\cdots + \alpha_n \frac{j+a_n}{k}.
$$
Since $\frac{a_2}{k_2} \in \left\langle \frac{a_3}{k_2k_3}, \ldots,
\frac{a_n}{k_2k_3}\right\rangle$ and $\left(\frac{a_3}{k_2k_3}, \ldots,
\frac{a_n}{k_2k_3}\right)$ is a complete intersection, by induction,
we get that $\left(\frac{j}{k_2}, \frac{j+a_2}{k_2}, \ldots,
\frac{j+a_n}{k_n}\right)$ is a complete intersection for some $j \gg
0$. Therefore by Theorem \ref{complete-int}, $\aa + (\jj)$ is a complete
intersection. If $\gcd(a_1, \ldots, a_n) = k_1 \neq 1$, then we can divide
by $k_1$ to get a complete intersection sequence $\aa'$, apply the
first part to obtain a $j'$ such that $\aa' + (\jj')$ is a complete
intersection and then by multiplying by $k_1$ to conclude that
$\aa+(\jj)$ is a complete intersection.
\end{proof}

We now characterize the complete intersection sequences when $n = 3$. 
It is actually possible to formulate a similar result in
the general case, but it is highly complicated. Therefore, we stick to
the case of $n = 3$.
\begin{theorem}\label{crit}
Let $\aa + (\jj) = (j, j+a, j+b, j+c)$. Then $\aa + (\jj)$ is a
complete intersection for $j \gg 0$ if and only if there exist
non-negative integers $m, k, \beta, \gamma$ such that $j = cm$,
$k \neq 1$ and one of the follwing is satisfied:
\begin{enumerate}
  \item[(1a)] $\gcd(a, c) = k$ and
  \item[(1b)] $ka = \beta b + \gamma c$. 
\end{enumerate}
\centerline{OR}
  \begin{enumerate}
	\item[(2a)] $\gcd(b, c) = k$ 
	\item[(2b)] $kb = \beta a + \gamma(c)$ with $\beta +
	  \gamma \leq k$.
  \end{enumerate}
\end{theorem}
\begin{proof}
If $\aa + (\jj)$ is a complete intersection sequence for $j \gg 0$, then
by Theorem \ref{main}, one of the two sets of conditions are
satisfied. We now prove the converse. First assume that (1a) and
(1b) are true. Let $\alpha = k - (\beta + \gamma)$. 
Using (1a), we can write $j+b = \alpha
\left(\frac{j}{k}\right) + \beta \left(\frac{j+a}{k}\right) + \gamma
\left(\frac{j+c}{k}\right).$ Note that $\gcd\left(\frac{a}{k},
\frac{c}{k}\right) = 1$. Therefore by Theorem \ref{complete-int1},
we get that $\left(\frac{j}{k}, \frac{j+b}{k},
\frac{j+c}{k}\right)$ is a complete intersection if $j \gg 0$ and
$\frac{j}{k} = \frac{c}{k} m$ for some $m$. Therefore if $j \gg 0$
and $j = cm$, then $\left(\frac{j}{k}, \frac{j+b}{k},
\frac{j+c}{k}\right)$ is a complete intersection. Let $k' =
\gcd(a,b,c) = \gcd(k, a)$. Then we can write
$$
\aa + (\jj) = \frac{j+a}{k'} (k') \sqcup k\left(\frac{j}{k}, 
\frac{j+b}{k},  \frac{j+c}{k}\right)
$$
with $k' \mid k$, $k' \neq k$, $\gcd\left(\frac{j+a}{k'}, k\right) =
1$ and $\left(\frac{j}{k}, \frac{j+a+b}{k},  \frac{j+a+b+c}{k}\right)$
a complete intersection. Therefore $\aa + (j)$ is a complete
intersection. If we assume the second set of condition, then the proof
can be obtained by interchanging the role of $a$ and $b$.
\end{proof}

\section{Examples}
We conclude the article by giving some examples. In the first example,
we show the periodicity.
\begin{example}
Let $\aa = (11,16,28)$. Let $j = 28m$ for some $m > 1$. Then it can be
seen that $28m + 11 = 2 (7m) + (7m + 4) + (7m + 7)$ and that $(7m,
7m+4, 7m+7)$ is a complete intersection (here we need $m > 1$). 
Therefore $(28m, 28m+11, 28m+16, 28m+28)$ is a complete intersection
sequence. 
\end{example}
The next example shows that $CI(\aa)$ could be non-empty and finite.
\begin{example}\label{ex1}
Let $\aa = (3,8,20)$. For $j = 28$, we have $\aa + (\jj) =
(28,31,36,48)$ and it can be seen that $\aa + (\jj) = 31 ~(1) \sqcup 4
(7,9,12)$ is a complete intersection split. Therefore $\aa + (j)$ is a
complete intersection. Suppose $\aa + (\jj)$ is a complete intersection
for $j \gg 0$. Since $\gcd(3, 20) = 1$, the only possible split for $j
\gg 0$ is of the form
$$
\aa + (\jj) = \frac{j+3}{k'} ~(k') \sqcup 4\left(\frac{j}{4},
\frac{j+8}{4}, \frac{j+20}{4}\right),
$$
with $\alpha + \beta + \gamma = 4$. This gives us the equation,
$$
12 = 8 \beta + 20 \gamma.
$$
Since this does not have a non-negative integer solution, we arrive at
a contradiction. Therefore, $CI(\aa)$ is finite. This examples also
shows that taking $j > a_n$ is not enough.
\end{example}

The next example shows that converse of Corollary \ref{cor1} is not
always true even for $j > a_n^2$ and $j = a_nm$.
\begin{example}\label{ex2}
Let $\aa = (8,17,18)$. Then $\aa$ is a complete intersection sequence.
For $j \gg 0$ and $j = 18m$, the only possibility of a complete
intersection split is of the form
$$
\aa + (\jj) = j+17 ~(1) \sqcup 2\left(\frac{j}{2}, \frac{j+8}{2},
\frac{j+18}{2}\right)
$$
such that $j+17 = \alpha \frac{j}{2} + \beta \frac{j+8}{2} + \gamma
\frac{j+18}{2}$ with $\alpha + \beta + \gamma = 2$. Therefore, $17 =
4\beta + 9 \gamma = 17$ and $\beta + \gamma \leq 2$. Since this does
not have a non-negative integer solution, this does not occur.
Therefore, $\aa + (\jj)$ can not be a complete intersection for $j >
18^2$.
\end{example}

\vskip 2mm
\noindent
\textbf{Acknowledgement:} The work was done during the first author's visit to
University of Missouri. He was funded by the Department of Science and Technology, Government of India. He would like to express sincere thanks to the funding agency and also to the Department of Mathematics at University of Missouri for the great hospitality provided to him.

\bibliographystyle{amsplain}

\end{document}